\documentclass[11pt]{article}

\usepackage{graphicx}
\usepackage{epsfig}
\usepackage{epsf}
\usepackage{latexsym}    
\usepackage{bbm}
\usepackage{float}
\usepackage{fancyvrb}
\usepackage[small]{caption} 

\usepackage{amsthm}
\usepackage{amsfonts}
\usepackage{amssymb}
\usepackage{amsmath}

\newtheorem{theorem}{Theorem}
\newtheorem*{thma}{Theorem A}

\newtheorem{lemma}{Lemma}

\newcommand{\later}[1]{{}}
\newcommand{\old}[1]{{}}
\long\def\ignore#1{}

\newcommand{\RR}{\mathbb{R}} 

\def\rho{\varrho}

\newcommand{\per}{{\rm per}}

\def\etal{{et~al.}}

\def\ie{{i.e.}}

\title{Metric inequalities for polygons}

\author{
Adrian Dumitrescu\thanks{Department of Computer Science,
University of Wisconsin--Milwaukee, Email:~\texttt{dumitres@uwm.edu}.
Supported in part by NSF CAREER grant CCF-0444188 
and NSF grant DMS-1001667.}}

\begin{document}
\maketitle

\begin{abstract}
Let $A_1,A_2,\ldots,A_n$ be the vertices of a polygon with unit
perimeter, that is $\sum_{i=1}^n |A_i A_{i+1}|=1$. We derive 
various tight estimates on the minimum and maximum values of the sum of
pairwise distances, and respectively sum of pairwise squared distances
among its vertices. In most cases such estimates on these sums
in the literature were known only for convex polygons.

In the second part, we turn to a problem of Bra\ss\  regarding the
maximum perimeter of a simple $n$-gon ($n$ odd) contained in a disk of
unit radius. The problem was solved by Audet et al.~\cite{AHM09b}, 
who gave an exact formula. Here we present an
alternative simpler proof of this formula.
We then examine what happens if the simplicity condition is dropped,
and obtain an exact formula for the maximum perimeter in this case as well.

\bigskip
\textbf{\small Keywords}: 
Metric inequalities,
polygon,
perimeter, 
sum of distances.
\end{abstract}

\section{Introduction} \label{sec:intro}

Let $A_1,A_2,\ldots,A_n$ be the vertices of a possibly self-crossing 
polygon (\ie, closed polygonal chain) with unit perimeter in the
Euclidean plane. Here the perimeter is 
 $\per(A_1 A_2 \ldots A_n)=$ \linebreak
$\sum_{i=1}^n |A_i A_{i+1}|$, where indices are taken modulo $n$
(\ie, $A_{n+1}=A_1$). 
Let $s(n)$ be the infimum of the sum of pairwise distances among the $n$ vertices, 
and $s_c(n)$ be the same infimum for the case of convex polygons:

\begin{align} \label{E8}
s(n) &= \inf_{\per(A_1 A_2 \ldots A_n)=1} \ \sum_{i<j} |A_i A_j|. \\
s_c(n) &= \inf_
{\substack
{\per(A_1 A_2 \ldots A_n)=1\\
A_1 A_2 \ldots A_n {\rm \ convex}}}
\ \sum_{i<j} |A_i A_j|.
\end{align}

Larcher and Pillichshammer~\cite{LP08} proved that $s_c(n)$ grows
linearly in $n$, and more precisely, that $s_c(n) \geq \frac{n-1}{2}$.
Alternative proofs were recently given by Aggarwal~\cite{Ag10} 
and L\"{u}k\H{o}~\cite{Lu10}. 
We have nearly equality, if $A_1$ is
close to $(0,0)$ and the other $n-1$ vertices $A_i$ $(i>1)$ are all close
to $(\frac12,0)$. Hence $s_c(n) = \frac{n-1}{2}$, as
previously conjectured by Audet~\etal~\cite{AHM07}. 
Here we extend this result for arbitrary
polygons and show that $s(n)$ has a similar behavior.

\begin{theorem} \label{T1} 
For every $n \geq 3$, 
$ s(n) \geq \frac{n}{4}$.
For $n$ even equality holds; for $n$ odd, $s(n) \leq \frac{n+1}{4}$.
\end{theorem} 

Let now $S(n)$ be the supremum of the sum of pairwise distances among the
vertices, and $S_c(n)$ be the same supremum for the case of convex
polygons: 

\begin{align} \label{E9}
S(n) &= \sup_{\per(A_1 A_2 \ldots A_n)=1} \ \sum_{i<j} |A_i A_j|. \\
S_c(n) &= \sup_
{\substack
{\per(A_1 A_2 \ldots A_n)=1\\
A_1 A_2 \ldots A_n {\rm \ convex}}}
\ \sum_{i<j} |A_i A_j|.
\end{align}

Larcher and Pillichshammer~\cite{LP06} considered the following
generalization of the sum of pairwise distances, for which they proved:

\begin{thma} {\rm \cite[Theorem~1]{LP06}} 
Let $f \colon [0,1/2] \to \RR_0^+$ be a function such that 
$f(x)/x \leq 2 f(1/2)$. Then for any $n \geq 3$ and for any convex polygon
with $n$ vertices and unit perimeter we have
$$ \sum_{i<j} f(|A_i A_j|) \leq f\left(\frac{1}{2}\right) 
\left \lfloor \frac{n}{2} \right \rfloor 
\left \lceil \frac{n}{2} \right \rceil. $$ 
This bound is the best possible. 
\end{thma} 

By taking $f(x)=x$, it follows from Theorem~A 
(as proved in~\cite{LP06}) 
that $S_c(n)$ is quadratic in $n$, and more precisely, that $S_c(n) \leq  
\frac{1}{2} \left \lfloor \frac{n}{2} \right \rfloor \left \lceil \frac{n}{2}
\right \rceil $, as previously conjectured by Audet~\etal~\cite{AHM07}. 
An alternative proof was given recently by Aggarwal~\cite{Ag10} 
based on classical results of Altman~\cite{Al72} for convex polygons. 
We have nearly equality if 
$A_1,\ldots,A_{\lfloor n/2 \rfloor}$ are close to $(0,0)$
and $A_{\lfloor n/2 \rfloor +1},\ldots,A_n$ are close to
$(\frac{1}{2},0)$; see~\cite{AHM07}.  Hence the above upper bound 
is best possible, thus $S_c(n) =
\frac{1}{2} \left \lfloor \frac{n}{2} \right \rfloor \left \lceil \frac{n}{2}
\right \rceil $. Here we show that convexity can be dropped and the
same inequality holds for arbitrary (not necessarily convex, and
possibly self-crossing) polygons: that is, 
$S(n) \leq  \frac{1}{2} \left \lfloor \frac{n}{2} \right \rfloor \left \lceil
\frac{n}{2} \right \rceil $. 
This result has been also obtained recently by L\"{u}k\H{o}~\cite{Lu10}.
(Since both his proof as well as ours rely on the triangle inequality,
both work in any metric space.)

\begin{theorem} \label{T2} 
For every $n \geq 3$, 
$$ S(n) = 
\frac{1}{2} \left \lfloor \frac{n}{2} \right \rfloor \left \lceil
\frac{n}{2} \right \rceil. $$
\end{theorem} 

Next we consider the sum of \emph{squared} distances.
Let now $t(n)$ be the infimum of the sum of pairwise squared distances among the
vertices, and $t_c(n)$ be the same infimum for the case of convex polygons: 

\begin{align} \label{E10}
t(n) &= \inf_{\per(A_1 A_2 \ldots A_n)=1} \ \sum_{i<j} |A_i A_j|^2. \\
t_c(n) &= \inf_
{\substack
{\per(A_1,A_2,\ldots,A_n)=1\\
A_1 A_2 \ldots A_n {\rm \ convex}}}
\ \sum_{i<j} |A_i A_j|^2.
\end{align}

For convex polygons, it is known that $t_c(n)$ is linear in $n$.
The current best lower bound, $t_c(n) \geq \frac{2n}{3\pi^2}$, is due
to Januszewski~\cite{Ja11}. From the other direction, placing $A_1$ near $(0,0)$,
$A_2$ near $(\frac12,0)$ and the other $n-2$ points near the midpoint of $A_1
A_2$, all in convex position, shows that $t_c(n) \leq \frac{n}{8}$~\cite{No09}. 

For arbitrary polygons, it is easy to make a construction for which this sum 
converges to $1/4$ as $n$ tends to infinity. For even $n$, place the
odd vertices at $(0,0)$, 
and the even vertices at $(\frac{1}{n},0)$. Then
$ Z = \sum_{i<j} |A_i A_j|^2 = \frac{n^2}{4} \cdot \frac{1}{n^2} = \frac{1}{4}$. 
For odd $n$, place the odd vertices at
$(0,0)$, and the even vertices at $(\frac{1}{n-1},0)$. Then
$ Z = \frac{n^2-1}{4} \cdot \frac{1}{(n-1)^2} = \frac{1}{4} \cdot
\frac{n+1}{n-1} \to \frac {1}{4} $.
Here we obtain a lower bound that is off by a factor of $2$ (in the limit). 

\begin{theorem} \label{T3} 
For every $n \geq 3$, 
$$\frac18 \leq t(n) \leq \frac14 + o(1). $$
\end{theorem} 

Finally, let $T(n)$ be the supremum of the sum of pairwise squared
distances  among the vertices, and $T_c(n)$ be the same supremum for
the case of convex polygons:  

\begin{align} \label{E11}
T(n) &= \sup_{\per(A_1 A_2 \ldots A_n)=1} \ \sum_{i<j} |A_i A_j|^2. \\
T_c(n) &= \sup_
{\substack
{\per(A_1 A_2 \ldots A_n)=1\\
A_1 A_2 \ldots A_n {\rm \ convex}}}
\ \sum_{i<j} |A_i A_j|^2.
\end{align}

By taking $f(x)=x^2$, it follows from Theorem~A (see~\cite{LP06}) 
that $T_c(n)$ is quadratic in $n$, and more precisely, that 
$T_c(n) \leq \frac{1}{4} \left \lfloor \frac{n}{2} \right \rfloor
\left \lceil \frac{n}{2} \right \rceil $. 
See also~\cite{LP08} for a simpler proof of a slightly weaker upper
bound, $T_c(n) \leq n^2/16$.  
An easy construction~\cite{LP08} 
(mentioned earlier in connection to $S_c(n)$) with 
vertices $A_1,\ldots, A_{\lfloor n/2 \rfloor}$ near $(0,0)$ 
and $A_{\lfloor n/2 \rfloor +1},\ldots, A_n$ near $(1/2,0)$, 
all in convex position, shows that the inequality 
$T_c(n) \leq \frac{1}{4} \left \lfloor \frac{n}{2} \right \rfloor
\left \lceil \frac{n}{2} \right \rceil $ is tight:
$T_c(n) = \frac{1}{4} \left \lfloor \frac{n}{2} \right \rfloor
\left \lceil \frac{n}{2} \right \rceil $.
Again, here we show that convexity can be dropped and the same
inequality holds for arbitrary (not necessarily convex, and possibly
self-crossing) polygons. That is, 
$T(n) \leq \frac{1}{4} \left \lfloor \frac{n}{2} \right \rfloor
\left \lceil \frac{n}{2} \right \rceil $, and we obtain:

\begin{theorem} \label{T4} 
For every $n \geq 3$, 
$$ T(n) = \frac{1}{4} \left \lfloor \frac{n}{2} \right \rfloor
\left \lceil \frac{n}{2} \right \rceil . $$
\end{theorem} 

Both Theorem~\ref{T2} and Theorem~\ref{T4} follow from a more general
statement asserting that Theorem~A holds without the
convexity assumption. 

\begin{theorem} \label{T5} 
Let $f \colon [0,1/2] \to \RR_0^+$ be a function such that 
$f(x)/x \leq 2 f(1/2)$. Then for any $n \geq 3$ and for any polygon
with $n$ vertices and unit perimeter we have
$$ \sum_{i<j} f(|A_i A_j|) \leq f\left(\frac{1}{2}\right) 
\left \lfloor \frac{n}{2} \right \rfloor 
\left \lceil \frac{n}{2} \right \rceil. $$ 
This bound is the best possible. 
\end{theorem}

In the second part of the paper we turn to the following 
problem~\cite[p. 437]{BMP05} posed by Bra\ss: For $n \geq 5$ odd,
what is the maximum perimeter of a simple $n$-gon ($n$ odd) contained
in a disk of unit radius? The problem can be traced back to the
collection of open problems in~\cite{Be03}~(Problem 4, p.~449).
A first solution was found by Audet~\etal~\cite{AHM09b} 
(Theorem~\ref{T6} below). Subsequently, another solution that 
also works in the hyperbolic plane was offered by L\'angi~\cite{La10}. 
Here we give yet another alternative solution. 

As noted in~\cite{Be03,BMP05}, for even $n$, one can
come arbitrarily close to the trivial upper bound $2n$ by a simple
polygon whose sides go back and forth near a diameter of the disk, but
for odd $n$ this construction does not work. Let $\Omega$ be a disk of
unit radius, and let 
\begin{equation} \label{E12}
F(n) = \sup_
{\substack
{\{A_1,A_2,\ldots,A_n\} \subset \Omega \\
A_1 A_2 \ldots A_n {\rm \ simple}}}
\ \per(A_1 A_2 \ldots A_n).
\end{equation}
So trivially, $F(n)=2n$, for even $n$. 
Fortunately, an exact formula for $F(n)$ can be also determined for
odd $n$:

\begin{theorem} {\rm \cite{AHM09b}.} \label{T6} 
For every $n \geq 3$ odd, 
\begin{equation} \label{E16}
F(n) = 
\frac{\sqrt{8(n-2)^2 -2 +2 \sqrt{1+8(n-2)^2}} \cdot \left(\sqrt{1+8(n-2)^2}+3\right)}
{4(n-2)}. 
\end{equation}
\end{theorem} 

A natural question is: What happens if the simplicity condition is dropped?
As before (for simple polygons) for even $n$, one can
come arbitrarily close to the trivial upper bound $2n$; however, for
odd $n$, the construction described previously (with the sides which go back and
forth near a diameter of the disk) still does not work. Let
\begin{equation} \label{E18}
G(n) = \sup_
{\{A_1,A_2,\ldots,A_n\} \subset \Omega}
\ \per(A_1 A_2 \ldots A_n).
\end{equation}

Clearly, $G(n) \geq F(n)$ holds, so $G(n)=2n$, for even $n$. 
For odd $n$, we determine an exact formula for $G(n)$ as well:
\begin{theorem} \label{T7} 
For every $n \geq 3$ odd, 
\begin{equation} \label{E19}
G(n) = 2n \cos \frac{\pi}{2n}. 
\end{equation}
\end{theorem} 

\paragraph{Notation.} Throughout the paper, let $P$ be a polygon with
$n$ vertices and unit perimeter, and $V(P)=\{A_1,A_2,\ldots,A_n\}$
denote its vertex set.  
Let $\ell(s)$ denote the line containing a segment $s$. 
Let $x(p)$ and $y(p)$ stand for the $x$- and $y$-coordinates of a
point $p$. For brevity we denote the set $\{1,2,\ldots,n\}$ by $[n]$. 

\paragraph{Related problems and results.}
Various extremal problems on the sum of distances and respectively
squares of distances among $n$ points in $\RR^d$ have been 
raised over time. For instance, more than 30 years ago,
Witsenhausen~\cite{W74} has conjectured that the maximum sum of
squared distances among $n$ points in $\RR^d$, whose pairwise
distances are each at most $1$ is maximized when the points are
distributed as evenly as possible among the $d+1$ vertices of a
regular simplex of edge-length $1$. He also proved 
that this maximum is at most $\frac{d}{2(d+1)} n^2$, which 
verified the conjecture at least when $n$ is a multiple of $d+1$. 
The conjecture has been proved for the plane by
Pillichshammer~\cite{Pi00}, and subsequently in higher dimensions
by Benassi and Malagoli~\cite{BM08}.
See also~\cite{Ag10,AHM07,AHM09a,AHM09b,FT56,FT59,Ja11,No09,No11,Pi00,Pi01} for
related questions on the sum of pairwise distances. 
In the spirit of Theorems~\ref{T6} and~\ref{T7},
a mathematical puzzle from Winkler's collection~\cite[p.~114]{W07} asks for
the minimum area of a simple polygon with an odd number of sides, each
of unit length.

\section{Preliminaries} 
\label{sec:T5}

\medskip
The following simple fact is needed in the proof of Theorems~\ref{T5}. 

\begin{lemma} \label{L3} 
Given an arbitrary polygon $P= A_1 A_2 \ldots A_n$, let 
$Q= A_{i_1} A_{i_2} \ldots A_{i_k}$, where $i_1<i_2<\ldots<i_k$, 
$3 \leq k \leq n$, be a sub-polygon of it. Then $\per(Q) \leq \per(P)$.
\end{lemma} 
\begin{proof} 
By the triangle inequality, for any $j \in [k]$, we have
$$ |A_{i_j} A_{i_{j+1}}| \leq \sum_{r = i_j}^{i_{j+1}-1} |A_r A_{r+1}|. $$
By adding up the above inequality over all $j \in [k]$, yields
$\per(Q) \leq \per(P)$, as required.
\end{proof} 

We also need the following extension of Lemma~1 in~\cite{LP06}
to the non-convex case. Its proof remains the same, since it does not
use convexity; see~\cite{LP06}.

\begin{lemma} \label{L4} 
Let $f \colon [0,1/2] \to \RR_0^+$ be a function such that 
$f(x)/x \leq 2 f(1/2)$. Then for any $n \geq 3$ and for any polygon
with $n$ sides $a_1,\ldots, a_n$ and perimeter at most one, \ie,
$\sum_{i=1}^n a_i \leq 1$, we have
$$ \sum_{i=1}^n f(a_i) \leq 2 f\left(\frac12\right). $$
\end{lemma}

\paragraph{Proof of Theorem~\ref{T5}.}
(Sketch.) The proof method is identical to that employed by Larcher and
Pillichshammer~\cite{LP08} for the convex case; we give a sketch
for completeness, and we refer the reader to their paper for details.  

If $n$ is even, a set of $\binom{\lfloor n/2 \rfloor}{2}$
quadrilaterals $\{Q_{ij}\}$, 
each a subpolygon of $P$, and a set of $\lfloor n/2 \rfloor $
edges $\{E_i\}$ are defined~\cite{LP08} so that the edges of the
quadrilaterals $Q_{ij}$ and the edges $E_i$ form a partition of the edge set
$\{A_i A_j \ | \ i<j\}$ (each edge appears exactly once). 
If $n$ is odd, a set of $\binom{\lfloor n/2 \rfloor}{2}$
quadrilaterals $\{Q_{ij}\}$ and a set of $\lfloor n/2 \rfloor $
triangles $\{R_i\}$ are defined~\cite{LP08} 
(these quadrilaterals and triangles are subpolygons of $P$),
so that the edges of the quadrilaterals $Q_{ij}$ and of the triangles $R_i$ 
form a partition of the edge set $\{A_i A_j \ | \ i<j\}$ (each edge
appears exactly once).  

If $n$ is even, Lemma~\ref{L3} yields that $\per(Q_{ij}) \leq 1$, and
obviously $|E_i| \leq 1/2$ holds.
If $n$ is odd, Lemma~\ref{L3} yields that $\per(Q_{ij}) \leq 1$, and
$\per(R_i) \leq 1$. 
In each case ($n$ even or odd), by Lemma~\ref{L4}, one can now bound
from above the sum $\sum_{i<j} f(|A_i A_j|)$ by the same expression, 
$f(1/2) \lfloor n/2 \rfloor \lceil n/2 \rceil $, 
as required.

From the other direction, we have nearly equality if 
$A_1,\ldots,A_{\lfloor n/2 \rfloor}$ are close to $(0,0)$
and $A_{\lfloor n/2 \rfloor +1},\ldots,A_n$ are close to
$(\frac{1}{2},0)$.
\qed

\section{Sum of distances: proofs of Theorems~\ref{T1} and~\ref{T2}} 
\label{sec:T12}

\paragraph{Proof of Theorem~\ref{T1}.}
Let $p \in V(P)$ be an arbitrary vertex of $P$, and let 
$Z(p)= \sum_{q \in V(P)} |pq|$ be the sum of distances from $p$ to the
other vertices. The sum of pairwise distances $Z$ satisfies 
$ 2Z= \sum_{p \in V(P)} Z(p)$. 
By the triangle inequality, for any $i \in [n]$, we
have 
$$ |p A_i| + |p A_{i+1}| \geq |A_i A_{i+1}|. $$
By summing over $i \in [n]$, we get
$$ 2Z(p) \geq \sum_{i=1}^n |A_i A_{i+1}|=1. $$
By summing over $p \in V(P)$, we get
$4 Z \geq n$, or $Z \geq n/4$, as required. 

To see that this inequality is almost tight, construct a non-convex
polygon as follows: For even $n$, place the odd vertices at 
$(0,0)$, and the even vertices at $(\frac{1}{n},0)$. Then
$ Z = \frac{n^2}{4} \cdot \frac{1}{n} = \frac{n}{4}$. 
For odd $n$, place the odd vertices at
$(0,0)$, and the even vertices at $(\frac{1}{n-1},0)$. Then
$ Z = \frac{n^2-1}{4} \cdot \frac{1}{n-1} = \frac{n+1}{4}$. 
\qed

\paragraph{Proof of Theorem~\ref{T2}.}
By taking $f(x)=x$, it follows from Theorem~\ref{T5}
that $S(n) \leq \frac{1}{2} \left \lfloor \frac{n}{2} \right \rfloor
\left \lceil \frac{n}{2} \right \rceil $, as required.
From the other direction, we clearly have
$S(n) \geq S_c(n) =
\frac{1}{2} \left \lfloor \frac{n}{2} \right \rfloor \left \lceil \frac{n}{2}
\right \rceil $, and the equality  
$S(n) = \frac{1}{2} \left \lfloor \frac{n}{2} \right \rfloor 
\left \lceil \frac{n}{2} \right \rceil $ is proved.
\qed

\section{Sum of squared distances: proofs of
  Theorems~\ref{T3} and~\ref{T4}} \label{sec:T345} 

We will need the following simple inequality:
\begin{lemma} \label{L1} 
Let $AB$ be a segment of length $a$, and $O$ be any point in the 
plane. Then 
$$ OA^2+OB^2 \geq \frac{a^2}{2} + 2 y^2, $$
where $y$ is the distance from $O$ to the line $\ell(AB)$ determined by $A$ and
$B$. In particular, $ OA^2+OB^2 \geq a^2/2$. 
\end{lemma} 
\begin{proof} 
Let $M$ be the projection of $O$ onto $\ell(AB)$, and write
$a_1=MA$, and $a_2=MB$.
Clearly,
$$ OA^2+OB^2 = a_1^2 +y^2+a_2^2 +y^2 \geq \frac{a^2}{2} + 2 y^2, $$
as desired.
\end{proof}

\paragraph{Proof of Theorem~\ref{T3}.}
The proof follows the same line of argument as the proof of Theorem~\ref{T1}. 
Let $p \in V(P)$ be an arbitrary vertex of $P$, and let 
$Z(p)= \sum_{q \in V(P)} |pq|^2$ 
be the sum of squared distances from $p$ to the other vertices. 
The sum of squared pairwise distances $Z$ satisfies 
$ 2Z= \sum_{p \in V(P)} Z(p)$. 
By Lemma~\ref{L1}, for any $i \in [n]$, we
have 
$$ |p A_i|^2 + |p A_{i+1}|^2 \geq \frac{|A_i A_{i+1}|^2}{2}. $$
By summing over $i \in [n]$, we get
$$ 2Z(p) \geq \frac{1}{2} \sum_{i=1}^n |A_i A_{i+1}|^2. $$
The Cauchy-Schwarz inequality yields
$$ \sum_{i=1}^n |A_i A_{i+1}|^2 \geq 
\frac{1}{n} \left(\sum_{i=1}^n |A_i A_{i+1}|\right)^2 = \frac{1}{n} . $$
Hence $4Z(p) \geq \frac{1}{n}$ for any $p \in V(P)$. 
Summing up this inequality over all $p \in V(P)$ yields
$$ 8Z= 4\sum_{p \in V(P)} Z(p) \geq \sum_{i=1}^n \frac{1}{n} =1, $$
or $Z \geq1/8$, as required. 
\qed

\paragraph{Remark.}
Interestingly enough, besides the construction mentioned in the
Introduction (with the odd vertices near $(0,0)$
and the even vertices near $(\frac{1}{n},0)$ or $(\frac{1}{n-1},0)$ 
depending on whether $n$ is even or odd)
there is yet another construction for which the sum of the squares of
the distances is at most $1/4$ in the limit. For odd $n$, consider $n$
points evenly distributed on a circle of radius $r=1/(2n \cos \frac{\pi}{2n})$, 
and labeled from $1$ to $n$, say in clockwise order. 
The polygon $P$ is the thrackle which connects the point labeled $i$
with the point labeled $i+ \frac{n-1}{2}$ (as usually the indexes are
taken modulo $n$). It is easy to verify that $P$ has unit perimeter
(see also Theorem~\ref{T7}).  Write $Z= \sum_{i<j} |A_i A_j|^2$. 
We have
$$ Z=n \sum_{i=1}^{(n-1)/2} 4r^2 \sin^2 \frac{i \pi}{n}=
4nr^2 \left( \sum_{i=1}^{(n-1)/2} \sin^2 \frac{i \pi}{n} \right). $$
Setting $k=\frac{n-1}{2}$ and $\alpha=\frac{\pi}{n}$ in the 
trigonometric identity~\cite[p.~64]{SCY79}
$$ \sum_{i=1}^{k} \sin^2 [i \alpha]= \frac{k+1}{2} -
\frac{\sin [(k+1) \alpha] \cdot \cos [k \alpha]}{2 \sin \alpha} $$
yields
$$ \sum_{i=1}^{(n-1)/2} \sin^2 \frac{i \pi}{n} = \frac{n+1}{4} -
\frac{1}{4}=\frac{n}{4} \ \Longrightarrow \ 
Z= 4nr^2 \cdot \frac{n}{4} = \frac{n^2}{4n^2} \cdot \frac{1}{\cos^2 \frac{\pi}{2n}}= 
\frac {1}{4} \cdot \frac{1}{\cos^2\frac{\pi}{2n}} \xrightarrow[n \to \infty]{} 
\frac {1}{4} . $$
For even $n$, duplicate one point in the construction above, and
obtain a similar estimate.

\paragraph{Proof of Theorem~\ref{T4}.}
By taking $f(x)=x^2$, it follows from Theorem~\ref{T5}
that $T(n) \leq \frac{1}{4} \left \lfloor \frac{n}{2} \right \rfloor
\left \lceil \frac{n}{2} \right \rceil $, as required.
From the other direction, we clearly have
$T(n) \geq T_c(n) =
\frac{1}{4} \left \lfloor \frac{n}{2} \right \rfloor \left \lceil \frac{n}{2}
\right \rceil $, and the proof of Theorem~\ref{T4} is complete.
\qed

\section{Odd polygons: proofs of Theorems~\ref{T6} and~\ref{T7}} 
\label{sec:T67}

\paragraph{Proof of Theorem~\ref{T6}.} 
For $n=3$ it is easily seen that the extremal polygon is 
an equilateral triangle of side $\sqrt{3}$, so let $n \geq 5$. 
Put $a=a(n)=\sqrt{1+8(n-2)^2}$ and let $H(n)$ be the right hand side
of~\eqref{E16}.  Then $H(n)$ can be also written as
\begin{equation} \label{E25}
H(n)=\frac{[(a+1)^2-4]^{1/2} (a+3)}{4(n-2)}. 
\end{equation}

Clearly, we have $a \geq 2\sqrt{2}(n-2) \geq 2\sqrt{2} \geq 3/2$,
hence  $(a+1)^2-4=a^2+2a-3 \geq a^2$, and consequently 
$[(a+1)^2-4]^{1/2} \geq a$. We show that this implies the inequality
$H(n) \geq 2(n-2)+\frac{3\sqrt{2}}{2}$:
\begin{equation} \label{E26}
H(n) \geq \frac{a(a+3)}{4(n-2)} \geq
\frac{2\sqrt{2}(n-2)}{4(n-2)} \left(2\sqrt{2}(n-2)+3 \right)=
\frac{2\sqrt{2}(n-2)+3}{\sqrt{2}} = 2(n-2)+\frac{3\sqrt{2}}{2}. 
\end{equation}

Let $o=(0,0)$ be the center of $\Omega$ and $X$ denote
the horizontal diameter of $\Omega$. Let $P$ be an extremal (limit)
polygon. Note that $P$ may have overlapping edges but is otherwise 
non-crossing. We will show that $P$ is unique and $\per(P) = H(n)$.
We start with the upper bound $\per(P) \leq H(n)$.
Let $BC$  be a longest side of $P$ of length $|BC|=z \leq 2$. 
We can assume that $BC$ is horizontal. Label each side $A_i A_{i+1}$ by $1$ or $0$ 
depending on whether it goes from left to right, or from right to left
in the $x$-direction (vertical sides are labeled arbitrarily). Since
$n$ is odd, we can find two consecutive sides of $P$ with the same label, say
$1$: they form a (weakly) $x$-monotone path, $\sigma$, of two
edges. By relabeling the vertices (if necessary), 
we can assume that this path consists of the edges $A_1 A_2$ and $A_2
A_3$: $\sigma= A_1 A_2 A_3$. Let $L_1=|\sigma|=|A_1 A_2|+ |A_2 A_3|$,
and let $L_2$ be the total edge length of the other $n-2$ edges, so
that $\per(P)=L_1 +L_2 \leq |\sigma| + (n-2)z$.  

Since $z \leq 2$ we can write $z=2 \sin \alpha$, for some $\alpha \in
[0,\pi/2]$. We first note that if $z \leq \sqrt3$ the upper bound
$\per(P) \leq H(n)$ follows immediately.
Indeed: (i) for $n=5$, $\per(P) \leq 5 \cdot \sqrt3 =8.66\ldots \leq
H(5)=8.97\ldots$ and we are done; 
(ii) for $n=7$, $\per(P) \leq 7 \cdot \sqrt3 =12.12\ldots \leq
H(7)=12.92\ldots$ and we are done; 
(iii) for $n \geq 9$, by \eqref{E26}, 
$\per(P) \leq n\sqrt3 \leq 2(n-2)+ \frac{3\sqrt{2}}{2} \leq H(n)$, 
and we are also done.
Therefore we can assume that $z \geq \sqrt3 =2\sqrt{3}/2$, hence 
$z=2 \sin \alpha$, for some $\alpha \in [\frac{\pi}{3},\frac{\pi}{2}]$. 

\begin{lemma} \label{L5} 
If $z=2 \sin \alpha$, for some $\alpha \in [\frac{\pi}{3},\frac{\pi}{2}]$, 
then $L_1 \leq 4 \cos \frac{\alpha}{2}$.
\end{lemma} 
\begin{proof} 
Since the path $\sigma=A_1 A_2 A_3$ is $x$-monotone and the polygon
$P$ is non-crossing, vertical rays from interior points of $BC$ meet
$\sigma$ on the same side of $BC$, if at all.  
Assume without loss of generality that $\sigma$ lies above $BC$ in
this sense; see Fig.~\ref{f2}(left and center) for two examples.
We may also assume that $BC$ lies below $o$ (\ie, $y(B) \leq 0$), since if
there is a counterexample to the lemma with $BC$ above $o$ then there
is also one with $BC$ below $o$: translate $BC$ down by $2 y(B)$ 
(to a parallel position below $o$); observe that $BC$ is 
still a horizontal segment of length $z$ contained in $\Omega$,
and $\sigma$ lies above $BC$. 
\begin{figure} [htb]
\centerline{\epsfxsize=6.2in \epsffile{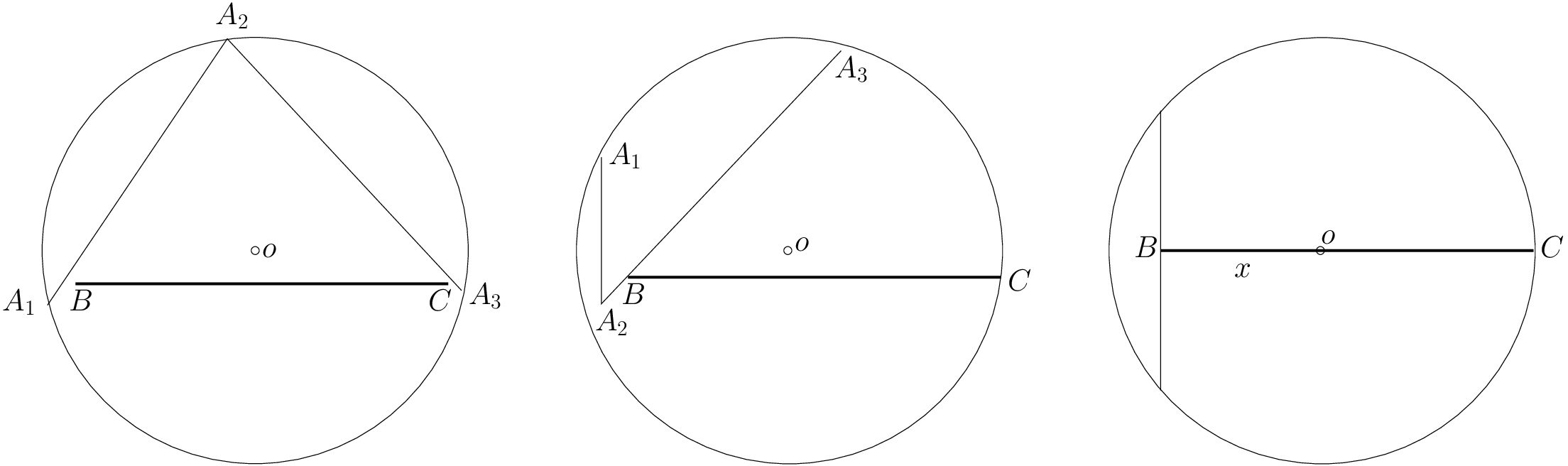}}
\caption{For a fixed length $z$, $L_1$ is maximum when $BC$ is a
chord of $\Omega$.}  
\label{f2}
\end{figure}

Consider for a moment the case when $BC$ is a right sub-segment of $X$ 
as in Fig.~\ref{f2}(right). Let $v$ be the length of the vertical chord 
incident to $B$. We have 
$$ v=2\sqrt{1-x^2}=2\sqrt{1-(z-1)^2}=2\sqrt{1-(2 \sin \alpha-1)^2}=
4\sqrt{\sin \alpha - \sin^2 \alpha}.
$$

We next verify that for $\alpha \in [\frac{\pi}{3},\frac{\pi}{2}]$ we have
\begin{equation} \label{E27}
z+v< 4 \cos \frac{\alpha}{2}, 
\end{equation}
or equivalently,
\begin{equation} \label{E28}
\sin \alpha + 2\sqrt{\sin \alpha - \sin^2 \alpha} < 2 \cos \frac{\alpha}{2}.
\end{equation}
Observe that both $f(\alpha)= \sin \alpha + 2\sqrt{\sin \alpha - \sin^2 \alpha}$ 
and $g(\alpha)= 2 \cos \frac{\alpha}{2}$ are decreasing functions on 
the interval $[\frac{\pi}{3},\frac{\pi}{2}]$. 
Partition the interval $[\frac{\pi}{3},\frac{\pi}{2}]$ into two
interior-disjoint intervals: 
$$ \left[\frac{\pi}{3},\frac{\pi}{2}\right]= 
[\alpha_1,\beta_1] \cup [\alpha_2,\beta_2], $$
where $\alpha_1= \pi/3$, $\beta_1=\alpha_2= 5\pi/12$, and $\beta_2=\pi/2$.
It is enough to check that $ f(\alpha_i) < g(\beta_i)$, for $i=1,2$:
$f(\alpha_1)=1.547\ldots < g(\beta_1)=1.586\ldots$,
and $f(\alpha_2)=1.328\ldots < g(\beta_2)=1.414\ldots$.
We have thereby verified \eqref{E28}. 

According to whether the slopes of $A_1 A_2$ and $A_2 A_3$ 
are $\geq 0$ or $\leq 0$, we say that the path $\sigma=A_1 A_2 A_3$ is of
type $++$, $+-$, $-+$, or $--$ (zero slopes are labeled arbitrarily). 
For example, $\sigma$ in Fig.~\ref{f2}(left) is of
type $+-$, while $\sigma$ in Fig.~\ref{f2}(center) is of type $-+$. 
We distinguish three cases:

\smallskip
{\em Case 1:} The path $A_1 A_2 A_3$ is of type $-+$, as in
Fig.~\ref{f2}(center) or in Fig.~\ref{f3}(left). 
If $x(A_2) \leq x(B)$
(the other case when $x(A_2) \geq x(C)$ is symmetric),
then $|A_1 A_2| \leq v$ and $|A_2 A_3| \leq z$, and inequality
\eqref{E27} concludes the proof. 
Assume now that $x(B) < x(A_2) < x(C)$, thus $A_2$ lies above $BC$.
The length of the chord extending $BC$ is $2 \sin \alpha_1$, for some
$\alpha_1$, where $\alpha \leq \alpha_1 \leq \pi/2$. 
\begin{figure} [htb]
\centerline{\epsfxsize=4.2in \epsffile{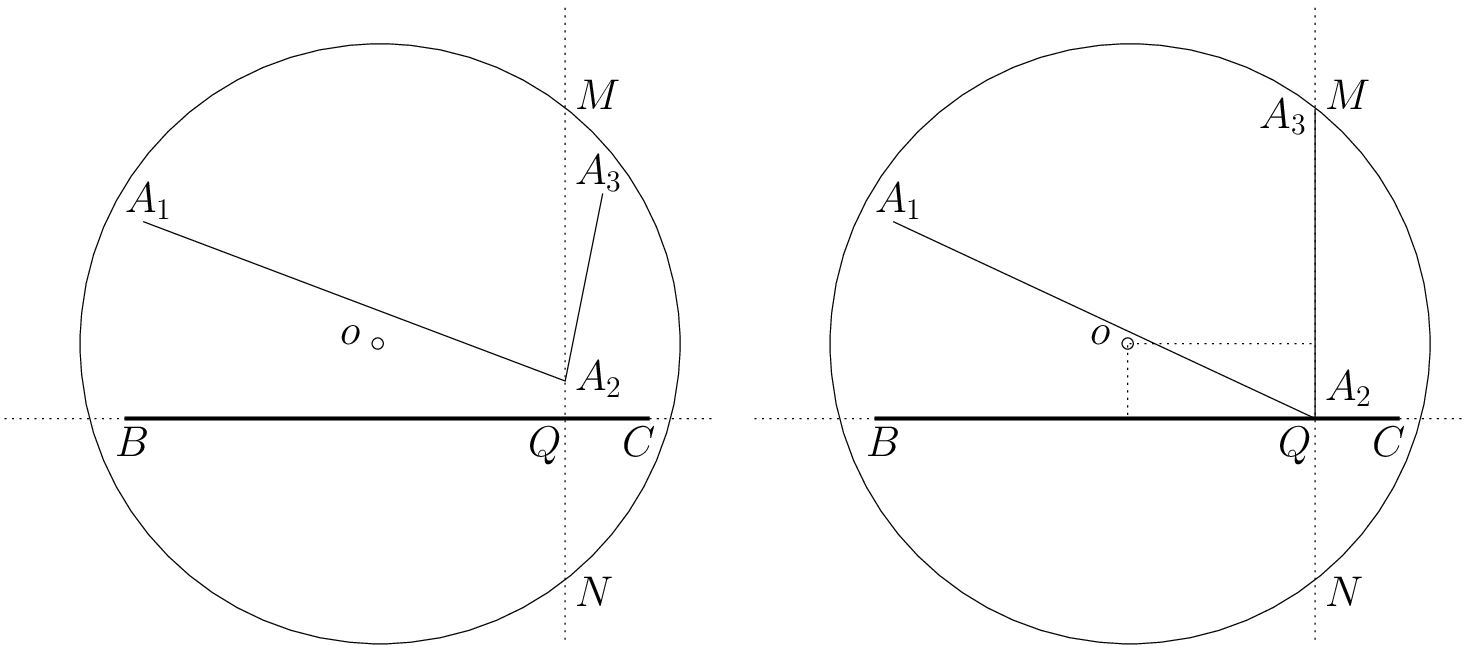}}
\caption{Case 1 in the proof of Lemma~\ref{L5}.}  
\label{f3}
\end{figure}
By symmetry we can assume that $x(o) \leq x(A_2)$, as in Fig.~\ref{f3}. 
Let $MN$ be the vertical chord through $A_2$, and $Q =BC \cap MN$.  
Assume that $MN$ subtends a central angle $2\beta$, for some 
$\beta \in [0,\pi/2]$. 
For fixed $\alpha$ and $\beta$, the length $|\sigma|$ is increased when 
$A_2$ is pushed down to $Q$, and $A_3$ is moved to $M$, \ie, 
$A_2=Q$ and $A_3=M$. 
The width and height of the rectangle with opposite vertices $o$ and
$Q$ are $\cos \beta$ and $\cos \alpha_1 \leq \cos \alpha$, respectively.
It follows that $|A_2 A_3| \leq \cos \alpha + \sin \beta$.
By the triangle inequality (used twice) we have
$$ |A_1 A_2| \leq |A_1 o| + |oQ| \leq |A_1 o| + \cos \alpha + \cos \beta \leq
1+ \cos \alpha + \cos \beta. $$
Recall the standard trigonometric inequality 
$\cos \beta + \sin \beta \leq \sqrt2$.  
Putting these together we obtain
\begin {align} \label{E1}
|\sigma| &=|A_1 A_2| + |A_2 A_3| \leq (1+ \cos \alpha + \cos \beta) +
(\cos \alpha + \sin \beta) \\ \nonumber
&= 1 + 2\cos \alpha +(\cos \beta + \sin \beta) 
\leq 1 + \sqrt2 + 2\cos \alpha. 
\end{align}
Since $\cos \alpha= 2 \cos^2 \frac{\alpha}{2} -1$, 
it remains to verify that 
$1 + \sqrt2 + 4\cos^2 \frac{\alpha}{2} -2 \leq 4 \cos \frac{\alpha}{2}$.
Make the substitution $t=\cos \frac{\alpha}{2}$; then 
$t \in [\cos \frac{\pi}{4}, \cos \frac{\pi}{3}]=
[\frac{\sqrt2}{2},\frac{\sqrt3}{2}]$,  
and we need to verify that 
$$ 1 + \sqrt2 + 4t^2 -2 \leq 4t, $$
or equivalently,
\begin{equation} \label{E2}
(2t-1)^2 \leq 2 -\sqrt2, \ \ {\rm for} \ \ t \in [\sqrt2/2,\sqrt3/2].
\end{equation}
It is easy to see that for the above range of $t$ we have
$$ (2t-1)^2 \leq (2 \sqrt3/2 -1)^2 = (\sqrt3 -1)^2 < 2 -\sqrt2, 
$$ 
as required.

\smallskip
{\em Case 2:} The path $A_1 A_2 A_3$ is of type $+-$, as in
Fig.~\ref{f2}(left). If $x(A_2) \leq x(B)$
(the other case when $x(A_2) \geq x(C)$ is symmetric),
then $|A_1 A_2| \leq v$ and $|A_2 A_3| \leq z$, and inequality
\eqref{E27} concludes the proof.  
Otherwise, $x(B) <x(A_2) < x(B)$, and we 
replace $A_1 A_2 A_3$ by a longer path as follows. 
If the extension of $A_2 A_1$ (beyond $A_1$) intersects $BC$, move 
$A_1$ to $B$; similarly if the extension of $A_2 A_3$ (beyond $A_3$)
intersects $BC$, move $A_3$ to $C$. Now $\sigma$ is still $x$-monotone 
and the extensions of $A_2 A_1$ and  $A_2 A_3$ meet $\partial \Omega$ 
without intersecting the interior of $BC$. Move $A_1$ and $A_3$ to
these intersection points on the circle $\partial \Omega$. 
Now move $A_2$ upward to the circle $\partial \Omega$ while increasing $|\sigma|$. 
We now have a $x$-monotone path $A_1 A_2 A_3$ of type $+-$ and above
$BC$ with all three points $A_1,A_2,A_3$ on the circle. 

Move $BC$ downward until it hits $\partial \Omega$,
and then rotate it around the endpoint on $\partial \Omega$; 
now $BC$ is a chord of length $z$ (subtending an angle of $2\alpha$
from the center $o$) below the chord $A_1 A_3$.  
Since $A_1$ and $A_3$ lie on the lower half-circle of $\partial
\Omega$, the chord $A_1 A_3$ subtends a central angle 
$2 \alpha_1$, where $\alpha \leq \alpha_1 \leq \pi/2$. 
For a fixed $\alpha$, $L_1=|\sigma|$ is maximized when the triangle
$\Delta{A_1 A_2 A_3}$ is isosceles with $|A_2 A_1|=|A_2 A_3|$,
$|A_1 A_3|= z = 2\sin \alpha$, and $o$ is in the interior of the triangle.
Indeed, $ L_1 = 2 (\sin \beta + \sin \gamma)$, 
where $2(\alpha_1+ \beta+ \gamma)=2 \pi$, thus
$$ L_1 =2 (\sin \beta + \sin \gamma) = 
4 \sin \frac{\beta+\gamma}{2} \cos\frac{\beta-\gamma}{2}
\leq 4 \sin \frac{\beta+\gamma}{2} = 4 \cos \frac{\alpha_1}{2}
\leq 4 \cos \frac{\alpha}{2}. $$
Observe that in the (unique) maximizing position,
$\beta=\gamma=\frac{\pi-\alpha}{2}$ and $o$ is in the interior of 
the triangle $\Delta{A_1 A_2 A_3}$.

\smallskip
{\em Case 3:} The path $\sigma=A_1 A_2 A_3$ is of type $++$ 
(or symmetrically, $--$). Write $\delta=\angle{A_1 A_2 A_3}$. 
Since $\sigma$ is also $x$-monotone, 
we have $\delta \geq \pi/2$. By the Cosine Theorem,
$$ |A_1 A_3|^2 = |A_1 A_2|^2 + |A_2 A_3|^2 -2 |A_1 A_2| |A_2 A_3| \cos \delta 
\geq |A_1 A_2|^2 + |A_2 A_3|^2. $$
Obviously, 
$ |A_1 A_2|^2 + |A_2 A_3|^2 \geq (|A_1 A_2|+|A_2 A_3|)^2 /2$, hence 
$$ |\sigma|= |A_1 A_2|+|A_2 A_3| \leq |A_1 A_3| \sqrt2 \leq 2 \sqrt2 
\leq 4 \cos \frac{\alpha}{2}, \ \ {\rm for} \ \ \alpha \in 
\left[\frac{\pi}{3},\frac{\pi}{2}\right], $$
as required.

\medskip
This concludes our case analysis. If any increase had occurred, a simple polygon
whose perimeter is strictly larger than that of $P$ could be
constructed by taking the new path $A_1 A_2 A_3$ and then going back
and forth near $A_2 A_3$ with the remaining $n-2$ edges.  
However, this would contradict the fact that $P$ were an extremal polygon. 
Observe that $L_1 \leq 4 \cos \frac{\alpha}{2}$ can hold with equality
only in Case 2. This is clear for Case 1. Equality in Case 3 requires
$\alpha=\pi/2$, thus $z=2$; moreover, it requires $|\sigma|=2\sqrt2$ and 
$|A_1A_2|=|A_2A_3|=\sqrt2$ with one of the two segments vertical and the
other horizontal; however, the length of the vertical segment cannot
exceed $1$, which is a contradiction. 
We have thus shown that for a fixed length $z$,  $L_1$
is maximized when $BC$ (of length $z$) is a chord of $\Omega$, $A_1=B$, $A_3=C$, 
and $|A_2 A_1|=|A_2 A_3|$ with $A_1,A_2,A_3$ on the circle and $o$ in
the interior of $\Delta{A_1 A_2 A_3}$. 
\end{proof} 

\medskip
Since $z=2 \sin \alpha$ is the length of a longest side, by
Lemma~\ref{L5} we get 
\begin{equation} \label{E13}
F(n) \leq L_1 + (n-2)z \leq 4 \cos \frac{\alpha}{2} + 2(n-2) \sin \alpha.
\end{equation}
We are thus led to maximizing the following function of one variable
$\alpha \in [0,\pi/2]$:
$$ f(\alpha)= 4 \cos \frac{\alpha}{2} + 2(n-2) \sin \alpha. $$
The function $f(\cdot)$ is maximized at the root of the derivative:
$$ f'(\alpha)= -2 \sin \frac{\alpha}{2} + 2(n-2) \cos \alpha. $$
Making the substitution $x= \sin \frac{\alpha}{2}$, 
and using the trigonometric identity 
$\cos \alpha= 1 - 2 \sin^2 \frac{\alpha}{2}$, yields the quadratic
equation in $x$:
$$ -2x + 2(n-2) (1-2x^2)=0, \ {\rm or} $$
$$ 2(n-2)x^2 +x -(n-2) =0. $$
The solution (corresponding to $\alpha \in [0,\pi/2]$) is
\begin{equation} \label{E14}
x= \sin \frac{\alpha}{2}=
\frac{-1 + \sqrt{1+8(n-2)^2}}{4(n-2)}. 
\end{equation}

This implies 
$$ \cos \frac{\alpha}{2}= \sqrt{1- \sin^2 \frac{\alpha}{2}} =
\frac{ \sqrt{8(n-2)^2 -2 +2 \sqrt{1+8(n-2)^2}}}{4(n-2)}. $$
Consequently, $F(n)$ is bounded from above by 
the maximum value of $f(\cdot)$, namely
\begin{align} \label{E15}
F(n) &\leq 4 \cos \frac{\alpha}{2} + 2(n-2) \sin \alpha =
4 \cos \frac{\alpha}{2} \left( (n-2) \sin \frac{\alpha}{2} +1 \right)
\nonumber \\
&= \frac{\sqrt{8(n-2)^2 -2 +2 \sqrt{1+8(n-2)^2}} \cdot 
\left(\sqrt{1+8(n-2)^2}+3\right)} {4(n-2)} =H(n). 
\end{align}

To see that this upper bound is tight construct a simple polygon as
follows. Let $A_1 A_3$ be a horizontal chord of length $z=2 \sin \alpha$,
below the center $o$, with $\alpha$ set according to \eqref{E14}. 
Let $A_2$ be the intersection point above $A_1 A_3$ of the vertical
bisector of $A_1 A_3$ with the unit circle $\partial \Omega$. The
remaining $n-2$ sides of the polygon go back and forth near the
horizontal chord $A_1 A_3$. Thus formula \eqref{E16} holds for every odd $n \geq 3$. 
This concludes the proof of Theorem~\ref{T6}. 
\qed

\paragraph{Proof of Theorem~\ref{T7}.} 
We start with the upper bound on $G(n)$. 
Consider the set of $n$-gons contained in $\Omega$, where each such
$n$-gon is given by the $n$-tuple of its vertex coordinates. 
Note that this forms a compact set, hence there exists an 
extremal polygon $P=A_1 \ldots A_n$, where $(A_{n+1}=A_1)$ which
attains the maximum perimeter. Observe two properties of $P$ that we
justify below: 
\begin{itemize}
\item Each vertex of $P$ lies on $\partial \Omega$. 
\item All sides of $P$ have equal length $<2$. 
\end{itemize}

First, assuming that $A_i$ lies in the interior of $\Omega$,
$\per(P)$ could be increased by moving $A_i$ orthogonally away from 
$A_{i-1} A_{i+1}$, or away from $A_{i-1}$ in case $A_{i-1} = A_{i+1}$. 
This would contradict the maximality of $P$, hence all 
vertices of $P$ lie on the circle.
Second, assume now that $A_{i-1}, A_i, A_{i+1}$ lie on the circle 
$\partial \Omega$, and $|A_{i-1} A_i| \neq |A_i A_{i+1}|$. Then
$\per(P)$ could be increased by moving $A_i$ on the circle and further
from $A_{i-1} A_{i+1}$ (to the midpoint of the arc). 
This again would contradict the maximality of $P$, hence 
all sides of $P$ are equal. Since $n$ is odd, it is obvious that the
common edge length is strictly smaller than $2$, since otherwise
$A_{n+1}$ cannot coincide with $A_1$. 

Having established the two properties above, we can now easily 
obtain an upper bound on the perimeter of $P$. 
Let $o$ be the center of $\Omega$. 
For each $i \in [n]$, label the side $A_i A_{i+1}$ by $+$ or $-$ 
depending on whether the center $o$ lies on the right of 
$\overrightarrow{A_i A_{i+1}}$ or on the left of $\overrightarrow{A_i A_{i+1}}$. 
This labeling encodes the winding of the edges of $P$ around the center $o$. 
Let $[n]=\Gamma_+ \cup \Gamma_-$ be the corresponding partition of $[n]$
determined by a positive or, respectively, negative labeling of $A_i
A_{i+1}$. Write $k=|\Gamma_+|$, and $l=|\Gamma_-|$, so $k+l=n$. 

We can assume that $l=0$; indeed if both $k>0$ and $l>0$, then there
exist two consecutive sides, $A_{i-1} A_i$ and $A_i A_{i+1}$, 
one with a positive label and one with a negative label. This implies
that $A_{i-1} = A_{i+1}$, hence $\per(P)$ could be increased 
(recall that the side length is smaller than $2$) by moving
$A_i$ to the point diametrically opposite to $A_{i-1}$ (and $A_{i+1}$), 
a contradiction. Hence $l=0$, $\Gamma_-=\emptyset$, and  $\Gamma= \Gamma_+=[n]$. 
For each $i \in [n]$, let $\angle{A_i  o A_{i+1}}=2\alpha$, 
where $0 < \alpha <\frac{\pi}{2}$. 
Since $P$ is a closed polygonal chain, $2n \alpha = m \pi$ 
for some positive integer $m$, $ 1 \leq m \leq n-1$. 
Consequently, the perimeter of $P$ is
\begin{equation} \label{E23}
\per(A_1 \ldots A_n)= 2n \sin \alpha = 2n \sin \frac{m \pi}{2n} 
\leq 2n \sin \frac{(n-1)\pi}{2n}= 2n \cos \frac{\pi}{2n}, 
\end{equation} 
as claimed. 

It remains to show that this bound can be attained. 
Consider $n$ points evenly distributed on the unit circle,
and labeled from $1$ to $n$, say in clockwise order. The polygon we need is the 
thrackle which connects the point labeled $i$ with the point labeled $i+
\frac{n-1}{2}$ (as usually the indexes are taken modulo $n$). 
It is easy to verify that its perimeter is given by the upper bound in
\eqref{E23}, and this concludes the proof of Theorem~\ref{T7}.
\qed

\medskip
\noindent {\bf Remarks.} For instance, $F(3)=3 \sqrt{3}=5.19\ldots$ corresponds
to an equilateral triangle of side $\sqrt{3}$, and
$ F(5)= \sqrt{70 + 2 \sqrt{73}} \cdot (\sqrt{73}+3)/12= 8.9774\ldots$. 
The exact formula \eqref{E16} easily yields an approximation of the form: 
$$ F(n)= 2(n-2) + 2 \sqrt{2} + O\left(\frac{1}{n}\right). $$

Note that the sum of the first two terms in this formula,
$ 2(n-2) + 2 \sqrt{2}$, gives (in the limit) the perimeter of a simple polygon 
whose first two sides have length $\sqrt{2}$ each, and whose remaining $n-2$ sides
go back and forth near a diameter of the unit disk. Thus the perimeter
of the extremal polygon in Theorem~\ref{T6} exceeds the perimeter of
the polygon described above only by a term that tends to zero with $n$.

It is interesting to observe that (for odd $n$) in contrast to $F(n)$,
$G(n)$ does get arbitrarily close to $2n$, as $n$ tends to infinity; that is,
$G(n)=2n-o(1)$. Indeed, the series expansion of $\cos x$ around $x=0$, 
$\cos x = 1-\frac{x^2}{2} + \ldots$ gives\footnote{%
A closed formula approximation avoiding the infinite sum is easily obtainable.}
$$ G(n)= 2n \cos \frac{\pi}{2n} = 
2n \left( 1 - \frac{\pi^2}{8n^2} + \ldots \right)= 
2n- \frac{\pi^2}{4n} + \ldots = 2n-o(1) . $$

\paragraph{Acknowledgement.}
The author is grateful to an anonymous reviewer for pointing out our
earlier proof of Theorem~\ref{T6} to be in error and for suggesting a
shorter proof of Theorem~\ref{T7}.

\end{document}